\documentclass[11pt]{article}

\usepackage[a4paper,margin=1in]{geometry}
\usepackage{amsmath,amssymb,amsthm,bm}
\usepackage{booktabs}
\usepackage{enumitem}
\usepackage{hyperref}
\usepackage{graphicx}
\usepackage{xcolor}
\usepackage{algorithm}
\numberwithin{equation}{section}

\newtheorem{theorem}{Theorem}[section]
\newtheorem{lemma}[theorem]{Lemma}
\newtheorem{proposition}[theorem]{Proposition}
\newtheorem{corollary}[theorem]{Corollary}
\newtheorem{assumption}[theorem]{Assumption}
\newtheorem{remark}[theorem]{Remark}

\newcommand{\D}{\mathbb D}
\newcommand{\C}{\mathbb C}
\newcommand{\R}{\mathbb R}
\newcommand{\dd}{\,d}
\newcommand{\del}{\partial}
\newcommand{\rank}{\operatorname{rank}}
\newcommand{\diag}{\operatorname{diag}}
\newcommand{\ord}{\operatorname{ord}}

\newcommand{\iu}{\mathrm{i}}
\newcommand{\e}{\mathrm{e}}

\title{A Hankel determinant zero-order principle for source counting in an inverse heat point-source problem}
\author{Zhiliang Deng\thanks{Corresponding author: School of Mathematical Science, University of Electronic Science and Technology of China, dengzhl@uestc.edu.cn}
\and Ailin Qian\thanks{Department of Mathematics and Statistics, Hubei  University of Science and Technology, junren751113@126.com}
\and Xiaomei Yang\thanks{School of Mathematics, Southwest Jiaotong University, yangxiaomath@swjtu.edu.cn}
}
\date{}

\begin{document}

\maketitle

\begin{abstract}
This paper studies the identification of an unknown number of stationary point sources in a two-dimensional heat equation from boundary flux measurements. Unlike many reconstruction approaches that assume the number of sources to be known in advance, we develop a determinant-based counting method that extracts this number directly from the measured data. In the unit disk, the Laplace-transformed and normalized boundary flux admits a Fourier moment representation whose low-frequency limit has a finite exponential-sum structure. This structure leads to a family of Hankel matrices and associated determinant characteristics. We prove that, under a generic determinant lifting condition, the vanishing order of the Hankel determinant at the zero Laplace frequency changes exactly when the Hankel order exceeds the true number of sources. Consequently, the source number is characterized by the first nonzero contour count of the determinant characteristic through the argument principle. We further establish a Rouch\'e-type stability result showing that the determinant zero count is preserved under sufficiently small perturbations induced by measurement noise, boundary discretization, and time truncation. After the source number is identified, the source locations and strengths are recovered from the low-frequency moment sequence by an annihilating-polynomial and Vandermonde reconstruction procedure. Numerical experiments confirm the predicted count pattern, demonstrate robustness with respect to contour selection, illustrate the role of the Rouch\'e margin under noise and near-degenerate configurations, and validate the subsequent recovery of source locations and strengths.
\end{abstract}

\section{Introduction}
\label{sec:introduction}

Inverse source problems arise in many scientific and engineering applications,
including medical imaging \cite{Baillet2001,Baratchart2005,He2018},
environmental monitoring \cite{Moghaddam2021}, groundwater contamination source
identification \cite{Badia2005,Kovalets2011,Shlomi2007}, and structural health
monitoring \cite{Gallet2022}. In such problems, one seeks to infer unknown
forcing terms from indirect observations of the corresponding physical field.
For parabolic equations, this task is particularly delicate because the heat
semigroup smooths the source information in both space and time. Consequently,
the recovery of source locations, intensities, temporal profiles, and the
number of active sources is typically ill-posed.

A substantial literature has studied uniqueness, stability, and reconstruction
for inverse source problems governed by parabolic and diffusion equations; see,
for example,
\cite{Bushuyev1995,Choulli1994,Isakov1990,Isakov1991,Kian2019,Reeve1994,Rundell1980,Solovev1989}
and the references therein. These works provide analytical foundations for
inverse source problems under different observation settings, source
structures, and a priori assumptions. Typical examples include inverse source
problems with separable or time-dependent source terms, as well as problems
with partial or sparse boundary observations; see, for instance,
\cite{Rundell2020}. Such results show that additional structural information on
the source is often essential for uniqueness and stable reconstruction in
parabolic inverse problems.

Numerical and statistical methods for parabolic inverse source problems have
also been developed from several perspectives. Optimization-based approaches
recover unknown source terms by minimizing data-misfit functionals, often
combined with regularization or optimal-control techniques; see, for example,
\cite{Hasanov2011}. Nguyen et al. \cite{Nguyen2019} proposed a
quasi-reversibility based numerical method for a parabolic inverse source
problem and applied it to a coefficient inverse problem. Lin et al.
\cite{Lin2022} developed a Bayesian sequential prediction method for a
semi-discrete parabolic source model with sparse boundary measurements, while
Lin et al. \cite{Lin2024} studied discontinuous source recovery using sparse
boundary flux data and dynamic sensors.

For heat equations with point sources, several recent works are more directly
related to the present study. Gong et al. \cite{Gong2026} considered the
recovery of one point source from flux data measured at several boundary
points. They proved unique recovery of the source location and a piecewise
constant-in-time amplitude in the unit ball, and also obtained uniqueness for a
compactly supported amplitude in two-dimensional simply connected smooth
bounded domains. Their proof combines eigenfunction representations,
heat- and Poisson-kernel estimates, and complex-analytic arguments. Gu et al.
\cite{Gu2025} studied the recovery of a Dirac point source from sparse boundary
measurements, established a uniqueness theorem, and proposed a least-squares
optimization method, solved by gradient descent, for reconstructing the source
location. For point-source identification with an unknown number of sources,
Deng et al. \cite{Deng2026} introduced a Bayesian thinning algorithm combining
a level set representation, pCN sampling, and a Poisson point process thinning
mechanism to infer the number, locations, and intensities of point sources from
boundary flux observations.

Motivated by these developments, we study a deterministic source-counting
problem for stationary point sources in the heat equation. Existing
sparse-measurement results for heat point sources mainly address one-source
uniqueness or fixed-dimensional location reconstruction, while Bayesian
thinning handles source-number uncertainty through a probabilistic sampling
mechanism. Here we seek a deterministic analytic principle for determining the
number of sources before carrying out location and strength reconstruction.

We work in the unit disk, where the Laplace-transformed and normalized
boundary flux admits an explicit Fourier moment representation. The key
observation is that the low-frequency boundary moment sequence has a finite
exponential-sum structure. This structure is reminiscent of classical
Prony-type methods and Hankel moment techniques for finite exponential sums
and sparse measure recovery
\cite{Castro2012,Potts2010,Potts2013,Potts2013_b}. However, the inverse heat
source problem considered here has an additional complex-frequency feature:
the moments depend analytically on the Laplace variable. This analytic
dependence allows us to determine the unknown source number from the zero
order of a Hankel determinant characteristic, rather than from the rank of a
static moment matrix.

%

More precisely, from the Fourier moments of the normalized boundary flux we
construct a family of Hankel matrices \(H_m(s)\) and define the determinant
characteristic
\[
\Delta_m(s)=\det H_m(s).
\]
At the zero Laplace frequency, \(H_m(0)\) inherits the finite-rank structure of
a finite exponential sum. We prove that, under a generic determinant lifting
condition, the vanishing order of \(\Delta_m(s)\) at \(s=0\) satisfies
\[
\operatorname{ord}_{s=0}\Delta_m(s)=0
\quad (m\leq N),
\qquad
\operatorname{ord}_{s=0}\Delta_m(s)=m-N
\quad (m>N).
\]
Thus the number of point sources is determined by the first Hankel order at
which this zero order becomes positive. The argument principle then turns this
characterization into a computable contour-counting formula.

We further analyze the stability of this determinant count. Rouch\'e's theorem
is used to show that the zero count is preserved under perturbations of the
determinant characteristic, provided the perturbation is smaller than the
determinant margin on the contour. We also trace how measurement noise,
boundary discretization, and finite-time truncation propagate from the boundary
flux data to the Fourier moments and then to the Hankel determinant. Once the
source number has been identified, the same low-frequency moment sequence is
used to recover the source locations and strengths through an
annihilating-polynomial and Vandermonde reconstruction procedure.

The numerical experiments confirm the predicted determinant count pattern,
show robustness with respect to the contour radius, and illustrate the role of
the Rouch\'e margin under noisy and near-degenerate configurations. They also
validate the subsequent recovery of source locations and strengths, while
showing that the strength reconstruction is more sensitive than the location
reconstruction due to the conditioning of the associated Vandermonde system.

The rest of the paper is organized as follows. Section~\ref{sec:for_bou}
formulates the inverse heat point-source problem in the unit disk and derives
the Laplace-transformed boundary flux representation for stationary-in-time
sources. Section~\ref{sec:hankel-counting} develops the Fourier moment
representation of the normalized boundary flux, constructs the Hankel
determinant characteristics, and establishes the source-counting formula based
on the vanishing order of \(\Delta_m(s)\) at \(s=0\). The same section also
gives the argument-principle interpretation of the determinant count and the
subsequent Prony--Vandermonde recovery of source locations and strengths.
Section~\ref{sec:noise-rouche-stability} analyzes discrete boundary data,
noise propagation, determinant perturbations, and the Rouch\'e-type stability
of the contour count. Section~\ref{sec:numerical} presents numerical
experiments validating the predicted count pattern, contour-radius robustness,
sensitivity to noise and near-degenerate source configurations, and the
recovery of source locations and strengths. Finally,
Section~\ref{sec:conclusion} concludes the paper and discusses possible
extensions.

\section{Problem Formulation}\label{sec:for_bou}
In this paper, we study the reconstruction of stationary-in-time point sources in a two-dimensional heat equation.  Let
$\D:=\{x\in\R^2:|x|<1\}$ be the unit disk. We consider the initial-boundary value problem:
\begin{align}
\label{eq:heat-source}
\left\{
\begin{aligned}
&\partial_t u(x,t)-\Delta u(x,t)=\sum_{j=1}^{N}q_j\delta(x-p_j), && (x,t)\in\D\times(0,\infty),\\
&u(x,t)=0, && (x,t)\in\partial\D\times(0,\infty),\\
&u(x,0)=0, && x\in\D.
\end{aligned}
\right.
\end{align}
Here $N\in\mathbb{N}^*$ is the unknown number of point sources, $p_j\in\D$ are the unknown source locations, $q_j\in\R\setminus\{0\}$ are the unknown source strengths, 
and $\delta(x-p_j)$ denotes the Dirac distribution centered at $p_j$.

Let $\theta_1, \ldots, \theta_L\in[0,2\pi)$ be a finite set of boundary observation angles, with corresponding boundary points $\e^{\iu\theta_\ell}$, $\ell=1, \ldots, L$. The available data are the boundary heat fluxes
\begin{align}
\label{eq:measurements}
f(\theta_\ell, t):=\frac{\partial u}{\partial\nu}(\e^{\iu\theta_\ell}, t), \quad t\in(0, T),
\quad \ell=1, \ldots, L,
\end{align}
where $\nu$ denotes the outward unit normal to $\partial\D$. The inverse source problem considered in this paper is to recover the number $N$, the locations $p_1, \ldots, p_N$ and the strengths $q_1, \ldots, q_N$ of point sources from the sparse boundary measurements $\{f(\theta_\ell, t)\}_{\ell=1}^{L}$.
This inverse problem is ill-posed in the sense that small perturbations in the measured data may lead to large errors in the reconstructed source configuration.


Applying the Laplace transform in time gives
\begin{align}
\label{eq:laplace-problem}
\left\{
\begin{aligned}
&s\widehat u(x, s)-\Delta\widehat u(x, s)
=
\frac1s\sum_{j=1}^{N}q_j\delta(x-p_j),
&&x\in\D,\\
&\widehat u(x, s)=0,
&&x\in\del\D.
\end{aligned}
\right.
\end{align}
Let \(G_s(x,p)\) denote the Dirichlet Green function of \(s-\Delta\) in \(\D\). Then the solution to \eqref{eq:laplace-problem} admits the form
\begin{align}
\label{eq:green-rep}
\widehat u(x, s)
=
\frac1s
\sum_{j=1}^{N}q_jG_s(x,p_j).
\end{align}
Consequently,  the Laplace-transformed boundary flux can be expressed as
\begin{align}
\label{eq:flux-laplace}
\widehat f(\theta,s)=\frac1s \sum_{j=1}^{N}q_j \frac{\del G_s}{\del\nu}(\e^{\iu\theta},p_j).
\end{align}
We remove the known factor \(1/s\) and define the normalized boundary flux:
\begin{align}
\label{eq:normalized-flux}
\mathcal F(\theta,s):=s\widehat f(\theta,s)=
\sum_{j=1}^{N}q_j\frac{\del G_s}{\del\nu}(\e^{\iu\theta}, p_j).
\end{align}

\begin{remark}
The complex variable in this paper is the Laplace variable \(s\). The modified Helmholtz Green function depends on \(\sqrt{s}\), but the zero-counting argument is formulated in the \(s\)-plane.
\end{remark}

Let
$$
p_j=\rho_j \e^{\iu\phi_j},
\qquad
0\leq\rho_j<1.
$$
For the unit disk, the boundary normal derivative of the Dirichlet Green function for \(s-\Delta\) has the Fourier--Bessel representation
\begin{align}
\label{eq:normal-green-fourier}
\frac{\del G_s}{\del\nu}(\e^{\iu\theta},p_j)=
-\frac1{2\pi}
\sum_{n=-\infty}^{\infty}
\frac{I_{|n|}(\sqrt{s}\rho_j)}{I_{|n|}(\sqrt{s})}
\e^{\iu n(\theta-\phi_j)},
\end{align}
where \(I_n\) is the modified Bessel function of the first kind.

\section{Hankel Determinant Characteristics and Source Counting}
\label{sec:hankel-counting}

In this section, we construct a family of Hankel determinant characteristics from the Fourier moments of the normalized boundary flux \eqref{eq:normalized-flux}. The key observation is that the low-frequency moment sequence has a finite exponential-sum structure. Consequently, the associated Hankel matrices have a finite-rank structure determined by the number of point sources. We then regard the Hankel matrices as analytic matrix-valued functions of the Laplace variable \(s\) and study the vanishing order of their determinants near \(s=0\). Under a natural non-degeneracy condition, this vanishing order encodes the source number and leads to an argument-principle source-counting formula. Once the source number has been identified, the same low-frequency moment sequence can be used to recover the source locations and strengths by an annihilating-polynomial method followed by a Vandermonde reconstruction.

For \(n=0, 1, \ldots\), define the positive Fourier moments of the normalized boundary flux of equation \eqref{eq:laplace-problem} by
\begin{align}
\label{eq:moment-def}
\mathcal{M}_n(s):=\frac1{2\pi}
\int_0^{2\pi}
\mathcal F(\theta,s)\e^{-\iu n\theta}\dd\theta.
\end{align}
Using \eqref{eq:normal-green-fourier}, we obtain
\begin{align}
\label{eq:moment-formula}
\mathcal{M}_n(s)
=-\frac1{2\pi}\sum_{j=1}^{N}
q_j\frac{I_n(\sqrt{s}\rho_j)}{I_n(\sqrt{s})}
\e^{-\iu n\phi_j}.
\end{align}
The ratio $\frac{I_n(\sqrt{s}\rho_j)}{I_n(\sqrt{s})}$ is analytic at \(s=0\). Indeed, using
$$
I_n(z)
=
\frac{1}{n!}
\left(\frac z2\right)^n
\left[
1+\frac{z^2}{4(n+1)}+O(z^4)
\right],
$$
we have
\begin{align}
\label{eq:bessel-ratio-expansion}
\frac{I_n(\sqrt{s}\rho_j)}{I_n(\sqrt{s})}
=
\rho_j^n
\left[
1+\frac{\rho_j^2-1}{4(n+1)}s+O(s^2)
\right].
\end{align}
Here \(s=0\) refers to the zero Laplace-frequency limit, whereas \(n\) denotes the Fourier mode index. 
Consequently,
\begin{align}
\label{eq:moment-zero}
\mathcal{M}_n(0)=-\frac1{2\pi}\sum_{j=1}^{N}q_j \left(\rho_j \e^{-\iu\phi_j}\right)^n.
\end{align}
Define
\begin{align}
\label{eq:lambda-weight}
\lambda_j:=\rho_j \e^{-\iu\phi_j}=\overline{p_j},
\qquad
w_j:=-\frac{q_j}{2\pi}.
\end{align}
Then the low-frequency moment sequence admits the finite exponential-sum representation
\begin{align}
\label{eq:finite-sum}
\mathcal{M}_n(0)=\sum_{j=1}^{N}w_j\lambda_j^n,
\qquad n=0,1,\ldots.
\end{align}
Here the number of exponential nodes is exactly the source number \(N\), the nodes \(\lambda_j=\overline{p_j}\) encode the source locations, and the weights \(w_j=-q_j/(2\pi)\) encode the source strengths.

For each \(m=1, 2, \ldots\), we define the \(m\times m\) Hankel matrix
\begin{align}
\label{eq:hankel-matrix}
H_m(s):=\bigl(\mathcal{M}_{r+c}(s)\bigr)_{r,c=0}^{m-1},
\end{align}
i.e.,
\begin{align}
H_m(s)=
\begin{pmatrix}
\mathcal{M}_0(s)&\mathcal{M}_1(s)&\cdots&\mathcal{M}_{m-1}(s)\\
\mathcal{M}_1(s)&\mathcal{M}_2(s)&\cdots&\mathcal{M}_m(s)\\
\vdots&\vdots&\ddots&\vdots\\
\mathcal{M}_{m-1}(s)&\mathcal{M}_m(s)&\cdots&\mathcal{M}_{2m-2}(s)
\end{pmatrix}.
\end{align}
The associated scalar characteristic function is then given by the Hankel determinant
\begin{align}
\label{eq:delta-def}
\Delta_m(s):=\det H_m(s).
\end{align}
At \(s=0\), using \eqref{eq:finite-sum}, we have
\begin{align}
\mathcal{M}_{r+c}(0)=\sum_{j=1}^{N}w_j\lambda_j^{r+c}.
\end{align}
Therefore
\begin{align}
\label{eq:hankel-factorization}
H_m(0)
=
V_m(\lambda)
\diag(w_1,\ldots,w_N)
V_m(\lambda)^{\top},
\end{align}
where
\begin{align}
\label{eq:vandermonde}
V_m(\lambda)
:=
\left(\lambda_j^r\right)_
{
r=0,\ldots,m-1;\ j=1,\ldots,N
}.
\end{align}

\begin{proposition}
\label{prop:rank-zero}
Assume that $\lambda_i\neq\lambda_j$, $i\neq j$,
and
$w_j\neq0$, $j=1, \ldots, N$.
Then
$$
\rank H_m(0)\leq N.
$$
Moreover, if \(m\geq N\), then
$$
\rank H_m(0)=N.
$$
In particular, for \(m=N\),
\begin{align}
\label{eq:deltaN-nonzero}
\Delta_N(0)
=
\left(\prod_{j=1}^{N}w_j\right)
\left[
\prod_{1\leq i<j\leq N}
(\lambda_j-\lambda_i)
\right]^2
\neq0.
\end{align}
\end{proposition}

\begin{proof}
The factorization \eqref{eq:hankel-factorization} gives
$$
\rank H_m(0)\leq N.
$$
If \(m\geq N\), then the Vandermonde matrix \(V_m(\lambda)\) has full column rank because the \(\lambda_j\)'s are distinct. Since \(\diag(w_1,\ldots,w_N)\) is invertible, it follows that
$$
\rank H_m(0)=N.
$$
For \(m=N\), taking determinants in \eqref{eq:hankel-factorization} gives
$$
\Delta_N(0)
=
\det V_N(\lambda)
\left(\prod_{j=1}^{N}w_j\right)
\det V_N(\lambda)^{\top}.
$$
Since
$$
\det V_N(\lambda)
=
\prod_{1\leq i<j\leq N}(\lambda_j-\lambda_i),
$$
we obtain \eqref{eq:deltaN-nonzero}.
\end{proof}


Proposition~\ref{prop:rank-zero} gives a static algebraic characterization of the source number through the rank of the low-frequency Hankel matrix \(H_m(0)\). To connect this finite-rank structure with a contour-counting principle in the complex Laplace domain, we next consider the full analytic Hankel matrix family \(H_m(s)\) and its determinant
$$
\Delta_m(s)=\det H_m(s).
$$
The rank deficiency of \(H_m(0)\) for \(m>N\) forces \(\Delta_m(0)=0\). Hence the relevant quantity is not merely whether \(\Delta_m(0)\) vanishes, but the order to which it vanishes at \(s=0\). This vanishing order will be used below to characterize the number of sources by the argument principle.

Since each moment \(\mathcal M_n(s)\) is analytic near \(s=0\), the determinant \(\Delta_m(s)\) is also analytic near \(s=0\). Thus there exist a nonnegative integer \(\tau_m\) and an analytic function \(A_m\), with \(A_m(0)\neq0\), such that
\begin{align}
\label{eq:delta-expansion}
\Delta_m(s)
=
s^{\tau_m}A_m(s).
\end{align}
We call
\begin{align}
\label{eq:tau-def}
\tau_m
:=
\operatorname{ord}_{s=0}\Delta_m(s)
\end{align}
the vanishing order of the Hankel determinant characteristic at \(s=0\).
%
For \(m\leq N\), one expects generically $\Delta_m(0)\neq0$,
and hence $\tau_m=0$.
For \(m>N\), Proposition~\ref{prop:rank-zero} implies
$\Delta_m(0)=0$.
The determinant then vanishes at least to order \(m-N\) under a generic rank-lifting mechanism.

\begin{assumption}[Generic determinant lifting]
\label{ass:det-lifting}
For every \(m>N\), the analytic perturbation \(H_m(s)-H_m(0)\) lifts the \(m-N\) dimensional null space of \(H_m(0)\) to first order. Equivalently,
\begin{align}
\label{eq:lifting-order}
\ord_{s=0}\Delta_m(s)=m-N.
\end{align}
\end{assumption}

\begin{remark}
Assumption~\ref{ass:det-lifting} is a transversality condition. It rules out accidental cancellations in the first nonzero term of the determinant expansion. Numerically, this condition is satisfied for generic source configurations.
\end{remark}

\begin{theorem}
\label{thm:source-number-zero-order}
Assume that the source locations are distinct, the source strengths are nonzero,
and Assumption~\ref{ass:det-lifting} holds. Then the vanishing order \(\tau_m\)
defined in \eqref{eq:tau-def} satisfies
\begin{align}
\label{eq:nu-pattern}
\tau_m
=
\begin{cases}
0, & m\leq N,\\
m-N, & m>N.
\end{cases}
\end{align}
Consequently, the number of point sources is characterized by
\begin{align}
\label{eq:N-from-nu}
N
=
\min\{m:\tau_m>0\}-1.
\end{align}
\end{theorem}
\begin{proof}
For \(m\leq N\), the generic non-degeneracy of the leading Hankel minors gives
$\Delta_m(0)\neq 0$.
Hence
$\tau_m=0$.
For \(m>N\), the determinant vanishes at \(s=0\) since
$$
\rank H_m(0)=N<m.
$$
By Assumption~\ref{ass:det-lifting}, the vanishing order is exactly
$$
\tau_m=m-N.
$$
This proves \eqref{eq:nu-pattern}. The identity \eqref{eq:N-from-nu} follows immediately because the first \(m\) for which \(\tau_m>0\) is \(m=N+1\).
\end{proof}


The preceding theorem characterizes the source number in terms of the
vanishing order \(\tau_m\). To obtain a computable formula for \(\tau_m\), we
now express it as a contour integral by means of the argument principle.
Let \(\Gamma_r\) be a positively oriented circle in the \(s\)-plane:
\begin{align}
\label{eq:circle-contour}
\Gamma_r:=\{s\in\C: |s|=r\},
\end{align}
where \(r>0\) is sufficiently small and \(\Delta_m\) has no zeros on \(\Gamma_r\). By the argument principle,
\begin{align}
\label{eq:arg-principle-delta}
\frac{1}{2\pi \iu}
\oint_{\Gamma_r}
\frac{\Delta_m'(s)}{\Delta_m(s)}d s
= \ord_{s=0}\Delta_m(s)=\tau_m,
\end{align}
provided \(\Gamma_r\) contains no zeros of \(\Delta_m\) other than the zero at \(s=0\).
Therefore, for each \(m\), we compute
\begin{align}
\label{eq:nu-contour}
\tau_m(r)
:=
\frac{1}{2\pi \iu}
\oint_{\Gamma_r}
\frac{\Delta_m'(s)}{\Delta_m(s)} ds.
\end{align}
For sufficiently small \(r\),
$$
\tau_m(r)=\tau_m.
$$
The source number is then estimated by
\begin{align}
\label{eq:N-estimator}
\widehat N=\min\{m:\tau_m(r)>0\}-1.
\end{align}

\begin{remark}
The contour integral in \eqref{eq:arg-principle-delta} is used to compute the zero order of the Hankel determinant characteristic at \(s=0\). Equivalently, one may compute the winding number of the curve \(\Delta_m(\Gamma_r)\) around the origin.
\end{remark}

Based on this winding-number interpretation, the source-counting procedure can be summarized as follows.

\paragraph{Source-counting procedure.}
For a prescribed maximum order \(m_{\max}\), compute
\[
\Delta_m^{(L)}(s)=\det H_m^{(L)}(s),
\qquad
m=1,\ldots,m_{\max},
\]
on the contour \(\Gamma_r\). The zero count is evaluated by the winding number of the closed curve
\[
\Delta_m^{(L)}(\Gamma_r)
=
\{\Delta_m^{(L)}(s):s\in\Gamma_r\}
\]
around the origin:
\[
\tau_m^{(L)}(r)
=
\frac1{2\pi}
\operatorname{Var}_{\Gamma_r}
\arg \Delta_m^{(L)}(s),
\]
where \(\operatorname{Var}_{\Gamma_r}\arg\) denotes the total continuous change of the argument along \(\Gamma_r\). Equivalently,
\[
\tau_m^{(L)}(r)
=
\frac{1}{2\pi\mathrm i}
\oint_{\Gamma_r}
\frac{(\Delta_m^{(L)})'(s)}
{\Delta_m^{(L)}(s)}\,ds.
\]
The number of point sources is estimated by
\[
\widehat N
=
\min\{m:\tau_m^{(L)}(r)>0\}-1.
\]
The estimate is accepted only when it is stable under small changes of the contour radius \(r\).

\section{Discrete Data, Noise Propagation, and Rouch\'e Stability}
\label{sec:noise-rouche-stability}

In this section, we analyze the stability of the Hankel determinant zero count under boundary discretization, finite-time observation, and measurement noise. The main purpose is to quantify how perturbations in the measured boundary flux propagate to the determinant characteristic $\Delta_m(s)=\det H_m(s)$
and to give a sufficient Rouch\'e condition under which the contour count remains unchanged.


In computations, the boundary flux is sampled at \(L\) uniformly distributed boundary points
$$
\theta_\ell=\frac{2\pi(\ell-1)}{L},
\qquad
\ell=1,\ldots,L.
$$
The continuous Fourier moment $\mathcal{M}_n(s)$ defined in \eqref{eq:moment-def} is approximated by the discrete Fourier sum
\begin{align}
\label{eq:discrete-moment}
\mathcal M_n^{(L)}(s)
=
\frac1L
\sum_{\ell=1}^{L}
\mathcal F(\theta_\ell,s)\e^{-\iu n\theta_\ell},
\qquad
n=0,1,\ldots,2m-2.
\end{align}
Accordingly, the  discrete Hankel matrix and determinant characteristic are defined by
\begin{align}
\label{eq:discrete-hankel}
H_m^{(L)}(s)
=
\bigl(\mathcal M_{a+b}^{(L)}(s)\bigr)_{a,b=0}^{m-1},
\end{align}
and
\begin{align}
\label{eq:discrete-delta}
\Delta_m^{(L)}(s)
=
\det H_m^{(L)}(s).
\end{align}
The discrete contour count is
\begin{align}
\label{eq:discrete-count}
\tau_m^{(L)}(r)
=
\frac{1}{2\pi\iu}
\oint_{\Gamma_r}
\frac{(\Delta_m^{(L)})'(s)}
{\Delta_m^{(L)}(s)}\,ds,
\end{align}
provided that \(\Delta_m^{(L)}\) has no zeros on \(\Gamma_r\). Equivalently, \(\tau_m^{(L)}(r)\) can be evaluated as the winding number of the closed curve
$$
\Delta_m^{(L)}(\Gamma_r)
=
\{\Delta_m^{(L)}(s):s\in\Gamma_r\}
$$
around the origin.

\subsection{From measurement noise to moment perturbation}

Let the measured boundary flux data be
\begin{align}
\label{eq:noisy-data}
f_\ell^\delta(t)
=
f(\theta_\ell,t)+\eta_\ell(t),
\qquad
\ell=1,\ldots,L,
\end{align}
where the noise satisfies
\begin{align}
\label{eq:noise-assumption}
\|\eta_\ell\|_{L^2(0,T)}
\leq
\delta,
\qquad
\ell=1,\ldots,L.
\end{align}
For \(s\in\Gamma_r\), define the truncated noisy Laplace transform by
\begin{align}
\label{eq:noisy-laplace}
\widehat f_\ell^{\delta,T}(s)
:=
\int_0^T \e^{-st}f_\ell^\delta(t)\,dt,
\end{align}
and the corresponding normalized data by
\begin{align}
\label{eq:noisy-normalized-data}
\mathcal F_\ell^{\delta,T}(s)
:=
s\widehat f_\ell^{\delta,T}(s).
\end{align}
The noisy discrete moments are then defined as
\begin{align}
\label{eq:noisy-discrete-moment}
\mathcal M_n^{\delta,L,T}(s)
:=
\frac1L
\sum_{\ell=1}^{L}
\mathcal F_\ell^{\delta,T}(s)\e^{-\iu n\theta_\ell}.
\end{align}


\begin{lemma}
\label{lem:data-to-moment-noise}
Assume that the noise terms \(\eta_\ell\) satisfy \eqref{eq:noise-assumption}.
Then, for every \(s\in\Gamma_r\),
\begin{align}
\label{eq:laplace-noise-bound}
\left|
s\int_0^T \e^{-st}\eta_\ell(t)\,dt
\right|
\leq
r\delta\sqrt{T}\,\e^{rT},
\qquad
\ell=1,\ldots,L.
\end{align}
Consequently, for all \(n=0,1,\ldots,2m-2\), the induced perturbation of the discrete moments satisfies
\begin{align}
\label{eq:moment-noise-bound}
\left|
\mathcal M_n^{\delta,L,T}(s)
-
\mathcal M_n^{L,T}(s)
\right|
\leq
r\delta\sqrt{T}\,\e^{rT},
\qquad
s\in\Gamma_r,
\end{align}
where \(\mathcal M_n^{L,T}(s)\) denotes the discrete moment computed from the noiseless truncated data.
\end{lemma}

\begin{proof}

By the Cauchy--Schwarz inequality,
\begin{align}
\left|
\int_0^T \e^{-st}\eta_\ell(t)\,dt
\right|
&\leq
\|\eta_\ell\|_{L^2(0,T)}
\left(
\int_0^T |\e^{-st}|^2\,dt
\right)^{1/2}  \notag\\
&=
\|\eta_\ell\|_{L^2(0,T)}
\left(
\int_0^T \e^{-2\Re s\,t}\,dt
\right)^{1/2}.
\end{align}
Since \(s\in\Gamma_r\), we have \(|s|=r\) and
\[
-\Re s\leq |s|=r.
\]
Hence
\[
\e^{-2\Re s\,t}
\leq
\e^{2rt},
\qquad
0\leq t\leq T.
\]
Therefore
\begin{align}
\left|
\int_0^T \e^{-st}\eta_\ell(t)\,dt
\right|
&\leq
\delta
\left(
\int_0^T \e^{2rt}\,dt
\right)^{1/2}  \notag\\
&\leq
\delta\sqrt{T}\,\e^{rT}.
\end{align}
Multiplying by \(|s|=r\), we obtain \eqref{eq:laplace-noise-bound}.
Using the definition \eqref{eq:noisy-discrete-moment} and the identity
\(|\e^{-\iu n\theta_\ell}|=1\), we obtain
\begin{align}
\left|
\mathcal M_n^{\delta,L,T}(s)
-
\mathcal M_n^{L,T}(s)
\right|
&\leq
\frac1L
\sum_{\ell=1}^{L}
\left|
s\int_0^T \e^{-st}\eta_\ell(t)\,dt
\right|  \notag\\
&\leq
r\delta\sqrt{T}\,\e^{rT}.
\end{align}
This proves the result.
\end{proof}

In addition to measurement noise, the computed moments also contain deterministic errors caused by boundary quadrature and finite-time truncation. For a fixed Hankel order \(m\), we assume the following aggregate perturbation bound:
\begin{align}
\label{eq:total-moment-error}
\max_{0\leq n\leq 2m-2}
\sup_{s\in\Gamma_r}
\left|
\mathcal M_n^{\delta,L,T}(s)-\mathcal M_n(s)
\right|
\leq
\eta_m(r,\delta,L,T).
\end{align}
A typical decomposition is
\begin{align}
\label{eq:eta-decomposition}
\eta_m(r,\delta,L,T)
\leq
r\delta\sqrt{T}\,\e^{rT}
+
\eta_{\rm disc}(m,L,r)
+
\eta_{\rm tr}(m,T,r),
\end{align}
where \(\eta_{\rm disc}\) denotes the boundary quadrature error and \(\eta_{\rm tr}\) denotes the deterministic time-truncation error. In the following analysis, only the aggregate bound \eqref{eq:total-moment-error} is needed.

\subsection{From moment perturbation to determinant perturbation}

Define the perturbed Hankel matrix by
$$
H_m^{\delta,L,T}(s)
=
\bigl(\mathcal M_{a+b}^{\delta,L,T}(s)\bigr)_{a,b=0}^{m-1},
$$
which corresponds to the ideal Hankel matrix \(H_m(s)\) introduced in Section~\ref{sec:hankel-counting}. The perturbed determinant characteristic is
\begin{align}
\label{eq:perturbed-delta}
\Delta_m^{\delta,L,T}(s)
:=
\det H_m^{\delta,L,T}(s).
\end{align}

\begin{lemma}[Moment noise to determinant noise]
\label{lem:moment-to-determinant}
Assume \eqref{eq:total-moment-error}. Let
\begin{align}
\label{eq:Bmr-def}
B_m(r)
:=
\sup_{s\in\Gamma_r}
\max\left\{
\|H_m(s)\|_2,\,
\|H_m^{\delta,L,T}(s)\|_2
\right\}.
\end{align}
Then, for all \(s\in\Gamma_r\),
\begin{align}
\label{eq:determinant-perturb-bound}
\left|
\Delta_m^{\delta,L,T}(s)-\Delta_m(s)
\right|
\leq
m^2 B_m(r)^{m-1}\eta_m(r,\delta,L,T).
\end{align}
\end{lemma}

\begin{proof}
Let
$$
E_m(s):=H_m^{\delta,L,T}(s)-H_m(s).
$$
By \eqref{eq:total-moment-error}, every entry of \(E_m(s)\) is bounded by
\(\eta_m(r,\delta,L,T)\). Therefore
\begin{align}
\label{eq:E-norm-bound}
\|E_m(s)\|_2
\leq
\|E_m(s)\|_F
\leq
m\,\eta_m(r,\delta,L,T).
\end{align}
We use the determinant Lipschitz estimate
\begin{align}
\label{eq:det-lipschitz}
|\det A-\det B|
\leq
m
\max\{\|A\|_2,\|B\|_2\}^{m-1}
\|A-B\|_2.
\end{align}
Applying \eqref{eq:det-lipschitz} with
$$
A=H_m^{\delta,L,T}(s),
\qquad
B=H_m(s),
$$
and using \eqref{eq:E-norm-bound}, we obtain
\begin{align}
\left|
\Delta_m^{\delta,L,T}(s)-\Delta_m(s)
\right|
&\leq
m B_m(r)^{m-1}\|E_m(s)\|_2  \notag\\
&\leq
m^2 B_m(r)^{m-1}\eta_m(r,\delta,L,T).
\end{align}
This proves the estimate.
\end{proof}

\subsection{Rouch\'e-stable determinant zero count}

Define the determinant margin on the contour by
\begin{align}
\label{eq:det-margin}
\mu_m(r)
:=
\min_{s\in\Gamma_r}
|\Delta_m(s)|.
\end{align}
We assume that \(\Delta_m\) has no zeros on \(\Gamma_r\), so that \(\mu_m(r)>0\).

\begin{theorem}[Noise-stable Hankel determinant count]
\label{thm:noise-stable-count}
Assume \eqref{eq:total-moment-error}. If
\begin{align}
\label{eq:noise-rouche-condition}
m^2 B_m(r)^{m-1}\eta_m(r,\delta,L,T)
<
\mu_m(r),
\end{align}
then \(\Delta_m^{\delta,L,T}\) and \(\Delta_m\) have the same number of zeros inside \(\Gamma_r\), counting multiplicities. Equivalently,
\begin{align}
\label{eq:noisy-count-equality}
\tau_m^{\delta,L,T}(r)
=
\tau_m(r),
\end{align}
where
\begin{align}
\label{eq:noisy-tau-def}
\tau_m^{\delta,L,T}(r)
=
\frac{1}{2\pi\iu}
\oint_{\Gamma_r}
\frac{(\Delta_m^{\delta,L,T})'(s)}
{\Delta_m^{\delta,L,T}(s)}\,ds.
\end{align}
\end{theorem}

\begin{proof}
By Lemma~\ref{lem:moment-to-determinant}, for all \(s\in\Gamma_r\),
\begin{align}
\left|
\Delta_m^{\delta,L,T}(s)-\Delta_m(s)
\right|
\leq
m^2 B_m(r)^{m-1}\eta_m(r,\delta,L,T).
\end{align}
Condition \eqref{eq:noise-rouche-condition} implies
\begin{align}
\left|
\Delta_m^{\delta,L,T}(s)-\Delta_m(s)
\right|
<
\mu_m(r)
\leq
|\Delta_m(s)|,
\qquad
s\in\Gamma_r.
\end{align}
Therefore, by Rouch\'e's theorem, \(\Delta_m^{\delta,L,T}\) and \(\Delta_m\) have the same number of zeros inside \(\Gamma_r\), counting multiplicities. The equality \eqref{eq:noisy-count-equality} follows from the argument principle.
\end{proof}

\begin{remark}
The sufficient condition \eqref{eq:noise-rouche-condition} has a direct interpretation. The quantity
$$
\mu_m(r)
=
\min_{s\in\Gamma_r}|\Delta_m(s)|
$$
is the Rouch\'e margin of the ideal determinant characteristic on the contour, whereas
$$
m^2 B_m(r)^{m-1}\eta_m(r,\delta,L,T)
$$
is an upper bound for the determinant perturbation caused by measurement noise, boundary discretization, and time truncation. The zero count is stable whenever the perturbation is smaller than this margin.
\end{remark}

\begin{corollary}[Stable source-number estimation]
\label{cor:stable-source-number}
Suppose that the ideal determinant counts satisfy
\begin{align}
\label{eq:ideal-count-pattern}
\tau_m(r)
=
\begin{cases}
0, & m\leq N,\\
m-N, & m>N,
\end{cases}
\end{align}
for \(m=1,\ldots,m_{\max}\). Assume that the Rouch\'e condition
\eqref{eq:noise-rouche-condition} holds for all \(m=1,\ldots,m_{\max}\). Then
\begin{align}
\label{eq:noisy-count-pattern}
\tau_m^{\delta,L,T}(r)=\tau_m(r),
\qquad
m=1,\ldots,m_{\max}.
\end{align}
Consequently, the source-number estimator
\begin{align}
\label{eq:noisy-N-estimator}
\widehat N^\delta
:=
\min\{m:\tau_m^{\delta,L,T}(r)>0\}-1
\end{align}
recovers the true source number:
\begin{align}
\label{eq:stable-N-result}
\widehat N^\delta=N.
\end{align}
\end{corollary}

\section{Numerical Experiments}
\label{sec:numerical}

In this section, we present numerical experiments to validate the proposed
Hankel determinant contour-counting method. The experiments are designed to
examine four aspects: the noiseless source-counting rule, robustness with
respect to the contour radius, sensitivity to noise and near-degenerate source
configurations, and the recovery of source locations and strengths after the
source number has been identified.

All numerical experiments are carried out in the unit disk. The boundary is
sampled uniformly, and the discrete Fourier moments, Hankel determinants, and
contour counts are computed according to the definitions in
Section~\ref{sec:hankel-counting}. Unless otherwise stated, the contour is the
circle
\[
\Gamma_r=\{s\in\mathbb C: |s|=r\},
\]
and the contour integral is evaluated by a uniform trapezoidal rule along
\(\Gamma_r\). The winding number is computed from the total continuous change
of the argument of the determinant curve. A contour count is regarded as
reliable only when the determinant curve stays away from the origin during the
winding-number computation.

In the noiseless experiments, the normalized boundary flux
\(\mathcal F(\theta,s)=s\widehat f(\theta,s)\) is evaluated from the
Fourier--Bessel representation of the Dirichlet Green function in the disk.
In the noisy transformed-data experiment, independent complex Gaussian noise is
added directly to the sampled values of \(\mathcal F(\theta_\ell,s_k)\). More
precisely, for a prescribed relative noise level \(\sigma\), we use
\[
\mathcal F^\delta(\theta_\ell,s_k)
=
\mathcal F(\theta_\ell,s_k)
+
\sigma\,\|\mathcal F(\cdot,s_k)\|_{\rm rms}\,\xi_{\ell k},
\]
where \(\xi_{\ell k}\) are independent standard complex Gaussian random
variables and \(\|\mathcal F(\cdot,s_k)\|_{\rm rms}\) denotes the root-mean-square
amplitude over the boundary samples. This noise model is used only as a stress
test for the determinant count, since it perturbs different complex frequencies
independently.

For each Monte Carlo experiment, the source configuration and the contour are
fixed, while the noise realization is resampled independently. The reported
success rate is the fraction of trials for which the estimated source number
\[
\widehat N=\min\{m:\tau_m^{(L)}(r)>0\}-1
\]
coincides with the true source number. For the reconstruction experiments, once
the source number is fixed, the locations and strengths are recovered from the
low-frequency moments by the annihilating-polynomial and Vandermonde procedure
described in Section~\ref{sec:hankel-counting}.


\subsection{Noiseless source counting and contour-radius robustness}
\label{subsec:noiseless-counting}

We first test the proposed source-counting rule in the noiseless case. The aim
is to verify the predicted count pattern
\[
\tau_m(r)
=
\begin{cases}
0, & m\leq N,\\
m-N, & m>N.
\end{cases}
\]
We consider source configurations with \(N=1,\ldots,5\). The contour radius is
first fixed at \(r=0.10\), and the Hankel order is varied from \(m=1\) to
\(m=8\).

Figure~\ref{fig:count-pattern} shows the resulting count sequences. In all
cases, the first nonzero count occurs exactly at \(m=N+1\). The observed
sequences are
\[
\begin{aligned}
N=1:\quad &(0,1,2,3,4,5,6,7),\\
N=2:\quad &(0,0,1,2,3,4,5,6),\\
N=3:\quad &(0,0,0,1,2,3,4,5),\\
N=4:\quad &(0,0,0,0,1,2,3,4),\\
N=5:\quad &(0,0,0,0,0,1,2,3).
\end{aligned}
\]
Therefore the estimator
\[
\widehat N=\min\{m:\tau_m(r)>0\}-1
\]
recovers the true source number for all tested configurations.

\begin{figure}[htbp]
\centering
\includegraphics[width=0.72\textwidth]{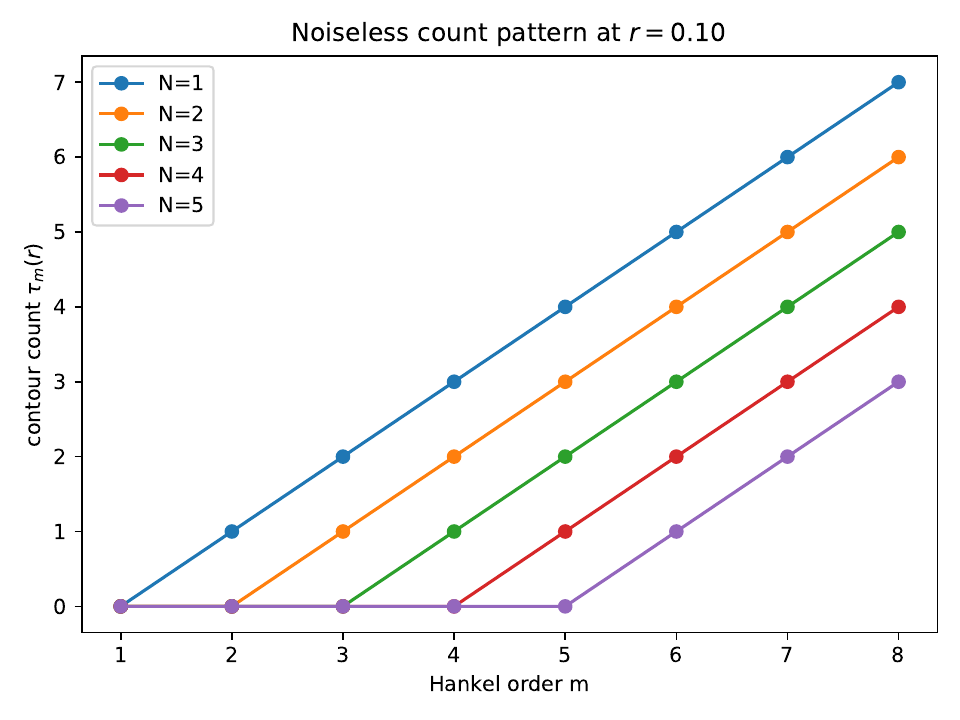}
\caption{Noiseless contour counts \(\tau_m(r)\) for \(N=1,\ldots,5\) at
\(r=0.10\). The first nonzero count occurs at \(m=N+1\).}
\label{fig:count-pattern}
\end{figure}

We next test the robustness of the contour count with respect to the radius
\(r\). For the case \(N=5\), the radius is varied over
\[
r=0.02,\ 0.05,\ 0.10,\ 0.20,\ 0.50.
\]
As shown in Figure~\ref{fig:radius-robustness}, the count sequence remains
\[
(\tau_1,\ldots,\tau_8)=(0,0,0,0,0,1,2,3)
\]
for all tested radii. Thus the source-counting rule is not tied to a single
particular contour radius. This supports the contour-validation strategy used
in the numerical implementation.

\begin{figure}[htbp]
\centering
\includegraphics[width=0.72\textwidth]{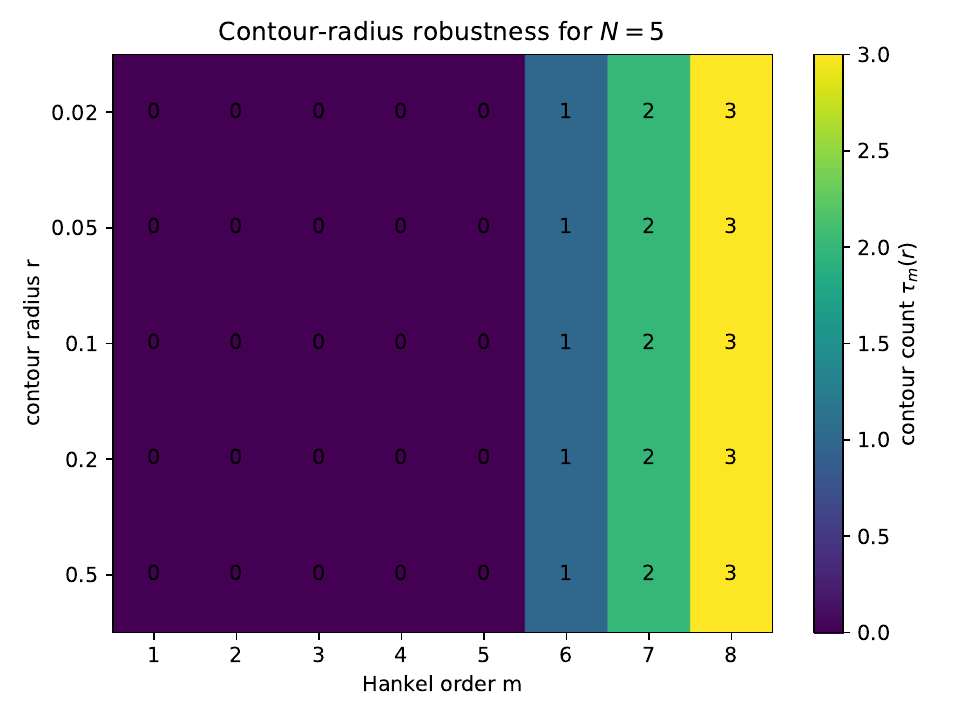}
\caption{Contour-radius robustness for \(N=5\). Each entry shows the contour count
\(\tau_m(r)\). The identical rows confirm that the source-counting pattern is
stable with respect to the contour radius.}
\label{fig:radius-robustness}
\end{figure}

Finally, Figure~\ref{fig:rouche-margin-first} reports the determinant margin
at the first nonzero determinant order,
\[
\mu_{N+1}(r)
=
\min_{s\in\Gamma_r}|\Delta_{N+1}(s)|.
\]
For each fixed \(N\), the margin increases mildly as \(r\) increases. However,
the margin decreases rapidly as the source number grows. For instance, at
\(r=0.10\), the values of \(\mu_{N+1}(r)\) are approximately
\[
2.18\times10^{-5},\quad
1.05\times10^{-7},\quad
1.66\times10^{-10},\quad
7.17\times10^{-14},\quad
1.10\times10^{-18}
\]
for \(N=1,\ldots,5\), respectively. This rapid decrease is consistent with the
increasing ill-conditioning of higher-order Hankel determinants and explains
why noisy source counting becomes more difficult as \(N\) increases.

\begin{figure}[htbp]
\centering
\includegraphics[width=0.72\textwidth]{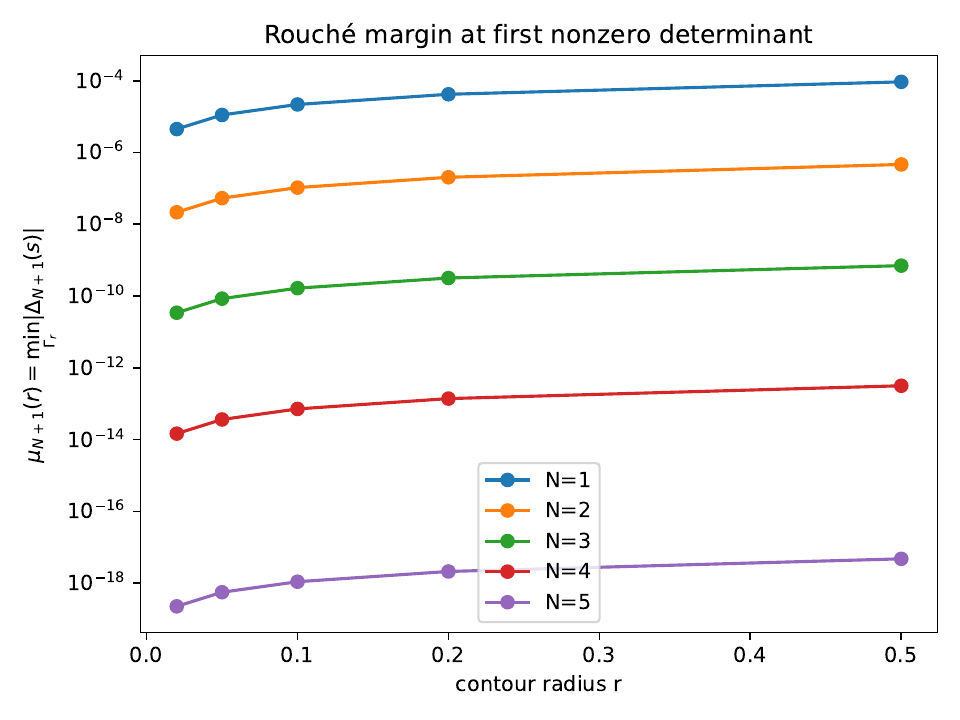}
\caption{Rouch\'e margin
\(\mu_{N+1}(r)=\min_{\Gamma_r}|\Delta_{N+1}(s)|\) at the first nonzero
determinant order. The margin decreases rapidly as \(N\) increases.}
\label{fig:rouche-margin-first}
\end{figure}

\subsection{Robustness under transformed-data noise}
\label{subsec:noise-robustness}

We next examine the sensitivity of the source-counting procedure to noisy data.
For each fixed source configuration and each noise level, 100 Monte Carlo
trials are performed. In each trial, independent complex Gaussian perturbations
are added to the sampled normalized boundary data \(\mathcal F(\theta_\ell,s_k)\).
The determinant contour-counting procedure is then applied to estimate the
source number. The success rate is defined as the percentage of trials for
which \(\widehat N=N\).

The noiseless case again gives a success rate of \(100\%\) for
\(N=1,2,3,4\). Under transformed-data noise, however, the success rate decreases.
The histograms in Figure~\ref{fig:noise-histograms} show that the dominant
failure mode is overestimation of the source number. For example, when the true
source number is \(N=2\), noisy trials frequently produce estimates
\(\widehat N=3,4,5\), or \(6\). A similar behavior is observed for \(N=3\).
For \(N=4\), underestimation also appears at higher noise levels.

Figure~\ref{fig:noise-success-rate} summarizes the Monte Carlo success rates.
The performance is strongest for small source numbers and deteriorates as both
the noise level and the source number increase. This behavior is consistent with
the Rouch\'e stability condition: the zero count is stable only if the
determinant perturbation remains smaller than the ideal determinant margin on
the contour. Since the determinant margin decreases rapidly with \(N\), the
higher-order tests are more sensitive to perturbations.

\begin{figure}[htbp]
\centering
\includegraphics[width=0.48\textwidth]{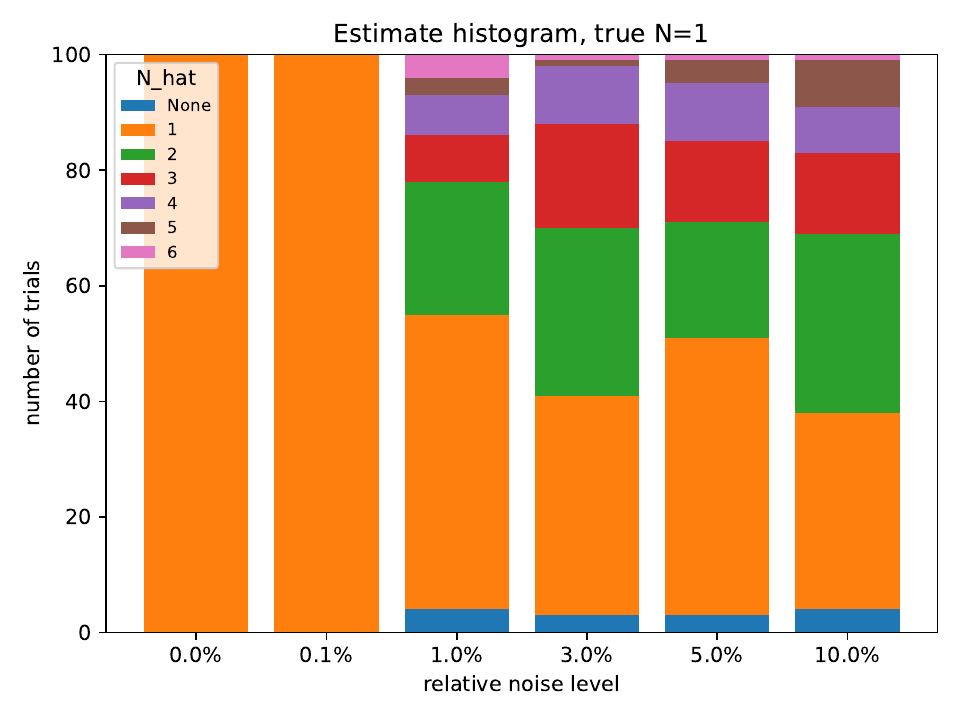}
\includegraphics[width=0.48\textwidth]{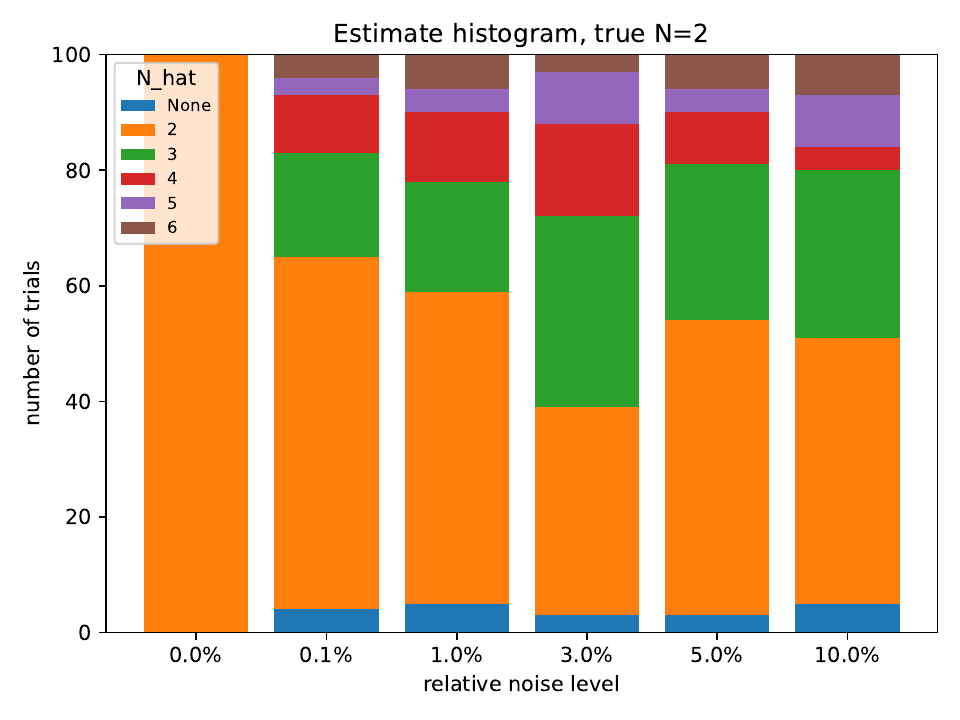}

\includegraphics[width=0.48\textwidth]{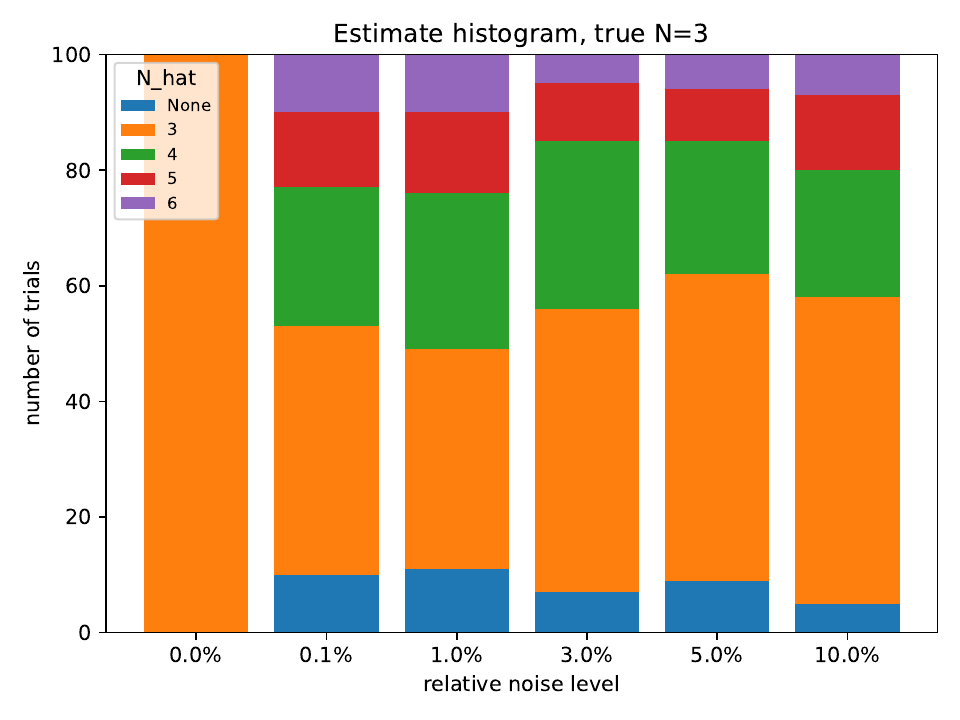}
\includegraphics[width=0.48\textwidth]{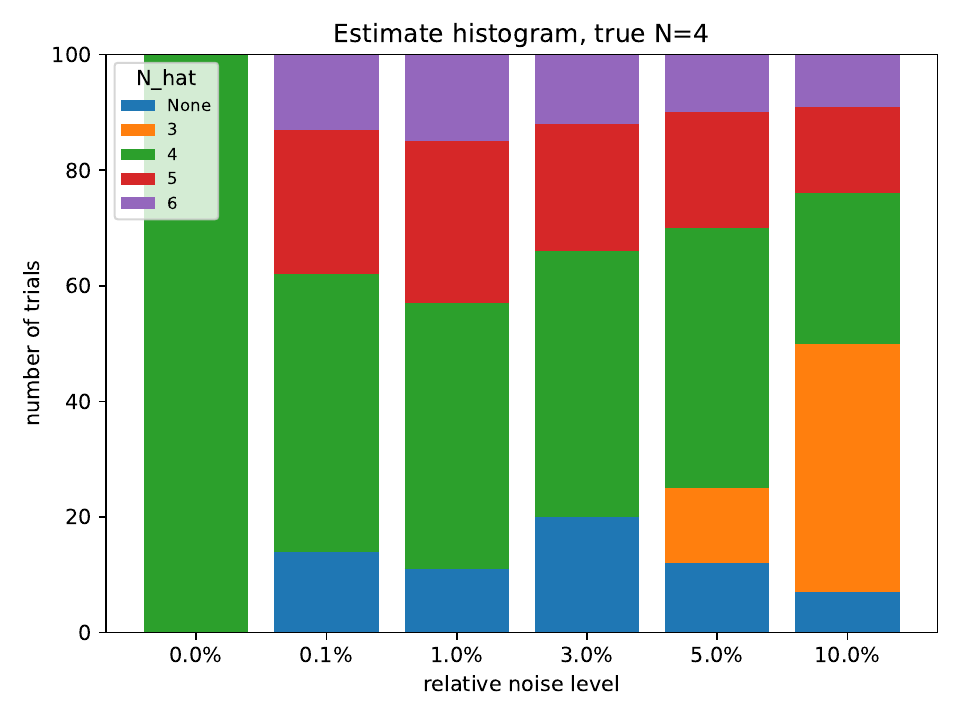}
\caption{Monte Carlo histograms of the estimated source number under noisy
transformed data for \(N=1,2,3,4\).}
\label{fig:noise-histograms}
\end{figure}

\begin{figure}[htbp]
\centering
\includegraphics[width=0.72\textwidth]{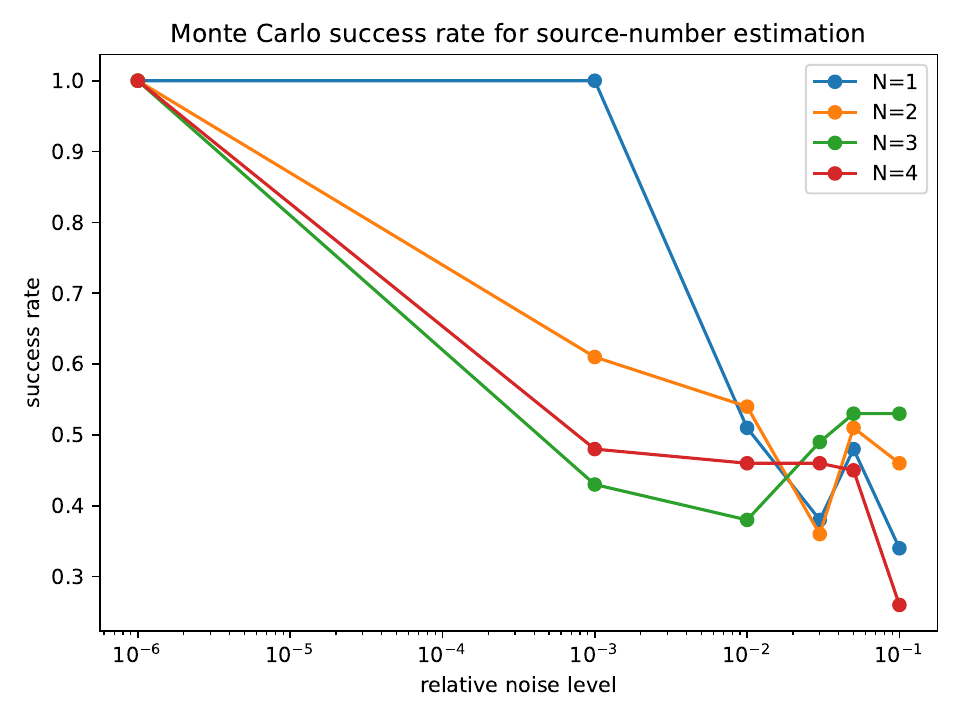}
\caption{Monte Carlo success rates for source-number estimation under noisy
transformed data. The success rate decreases as the noise level and the source
number increase.}
\label{fig:noise-success-rate}
\end{figure}

It is important to interpret this experiment as a stress test. The perturbations
are added independently to each complex transformed sample
\((\theta_\ell,s_k)\), which destroys part of the analytic correlation in the
Laplace variable \(s\). In a physical measurement model, noise is naturally
added in the time domain and then propagated through the Laplace transform.
Nevertheless, the present experiment is useful because it illustrates the role
of the determinant margin in the Rouch\'e condition.

\subsection{Effect of source separation and source strength}
\label{subsec:resolution-limits}

We now study how the determinant margin changes when the admissibility
conditions become nearly violated. Two situations are considered: nearly
colliding sources and weak sources. In both tests, the noiseless determinant
count remains correct. The main effect is instead observed in the Rouch\'e
margin
\[
\mu_{N+1}(r)
=
\min_{s\in\Gamma_r}|\Delta_{N+1}(s)|.
\]

In the close-source experiment, the true source number is \(N=2\). The distance
between two sources is varied from \(0.400\) to \(0.015\). For all tested
separations, the contour-count sequence remains
\[
(\tau_1,\ldots,\tau_7)=(0,0,1,2,3,4,5),
\]
and hence the method gives \(\widehat N=2\). However, the determinant margin
decreases rapidly as the two sources approach each other. More precisely,
\(\mu_3(r)\) decreases from approximately
\[
2.74\times10^{-8}
\]
at \(|p_1-p_2|=0.400\) to
\[
1.14\times10^{-12}
\]
at \(|p_1-p_2|=0.015\). Thus close sources do not destroy the noiseless count,
but they significantly reduce the admissible perturbation level.

The weak-source experiment shows a similar phenomenon. Here the true source
number is \(N=3\), and the third source strength \(q_3\) is varied from \(1\) to
\(0.01\). For all tested values of \(q_3\), the count sequence remains
\[
(\tau_1,\ldots,\tau_8)=(0,0,0,1,2,3,4,5),
\]
so that the method correctly returns \(\widehat N=3\). Nevertheless, the
margin \(\mu_4(r)\) decreases from approximately
\[
3.99\times10^{-10}
\]
when \(q_3=1\) to
\[
2.42\times10^{-12}
\]
when \(q_3=0.01\). Thus weak sources are detectable in exact noiseless moment
data, but their detection becomes increasingly sensitive to perturbations.

\begin{figure}[htbp]
\centering
\includegraphics[width=0.48\textwidth]{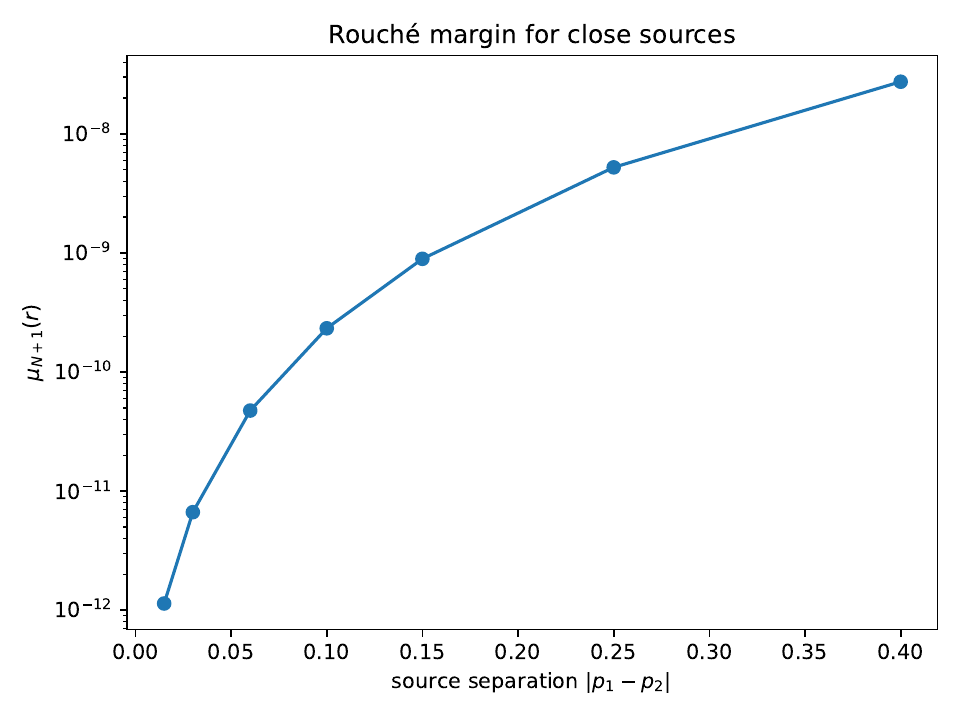}
\includegraphics[width=0.48\textwidth]{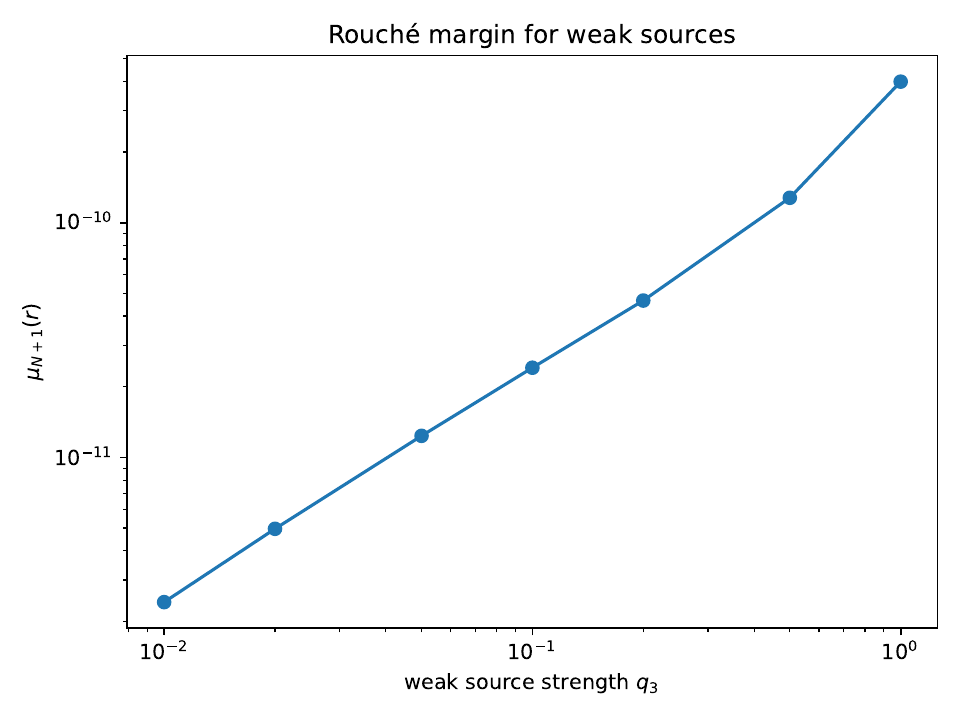}
\caption{Rouch\'e margins for close sources and weak sources. The margin
decreases as the source separation becomes smaller or as one source strength
tends to zero.}
\label{fig:resolution-margin}
\end{figure}

These experiments support the role of the admissibility assumptions. The source
separation condition prevents the determinant margin from becoming too small
due to nearly colliding exponential nodes, while the lower bound on source
strengths prevents weak sources from being masked by perturbations.

\subsection{Recovery of source locations and strengths}
\label{subsec:location-strength-recovery}


Finally, we test the recovery of source locations and strengths after the
source number has been identified. In this experiment, the true source number
\(N\) is used in the Prony--Vandermonde reconstruction step. By the finite
exponential-sum representation \eqref{eq:finite-sum} with the parametrization
\eqref{eq:lambda-weight}, the low-frequency moments determine the exponential
nodes \(\lambda_j\) and weights \(w_j\). The source locations are recovered
from the roots of the annihilating polynomial, and the strengths are then
obtained by solving the corresponding Vandermonde system.

The reconstruction errors are measured by
\[
E_p
=
\min_{\pi}
\left(
\frac1N
\sum_{j=1}^{N}
|\widehat p_j-p_{\pi(j)}|^2
\right)^{1/2},
\]
and
\[
E_q
=
\min_{\pi}
\frac{
\left(
\sum_{j=1}^{N}
|\widehat q_j-q_{\pi(j)}|^2
\right)^{1/2}
}{
\left(
\sum_{j=1}^{N}|q_j|^2
\right)^{1/2}
},
\]
where \(\pi\) ranges over all permutations of \(\{1,\ldots,N\}\).

Figures~\ref{fig:reconstruction-errors} reports the mean errors for
\(N=1,2,3,4\) under relative perturbations of the low-frequency moment
sequence. In the noiseless case, the reconstruction reaches machine precision.
For instance, when \(N=4\), the noiseless mean errors are approximately
\[
E_p=3.24\times10^{-16},
\qquad
E_q=1.30\times10^{-15}.
\]
As the moment noise level increases, both \(E_p\) and \(E_q\) increase nearly
linearly on the log--log scale. This indicates conditional stability of the
Prony--Vandermonde reconstruction for separated nodes and sufficiently small
moment perturbations.

The strength recovery is generally more sensitive than the location recovery.
For \(N=4\), the mean errors are approximately
\[
E_p=1.74\times10^{-5},
\qquad
E_q=1.44\times10^{-4}
\]
at relative moment noise \(10^{-6}\), while at relative moment noise \(10^{-2}\)
they increase to
\[
E_p=1.25\times10^{-1},
\qquad
E_q=8.03\times10^{-1}.
\]
This stronger sensitivity of the strength reconstruction is caused by the
conditioning of the Vandermonde system.

\begin{figure}[htbp]
\centering
\includegraphics[width=0.48\textwidth]{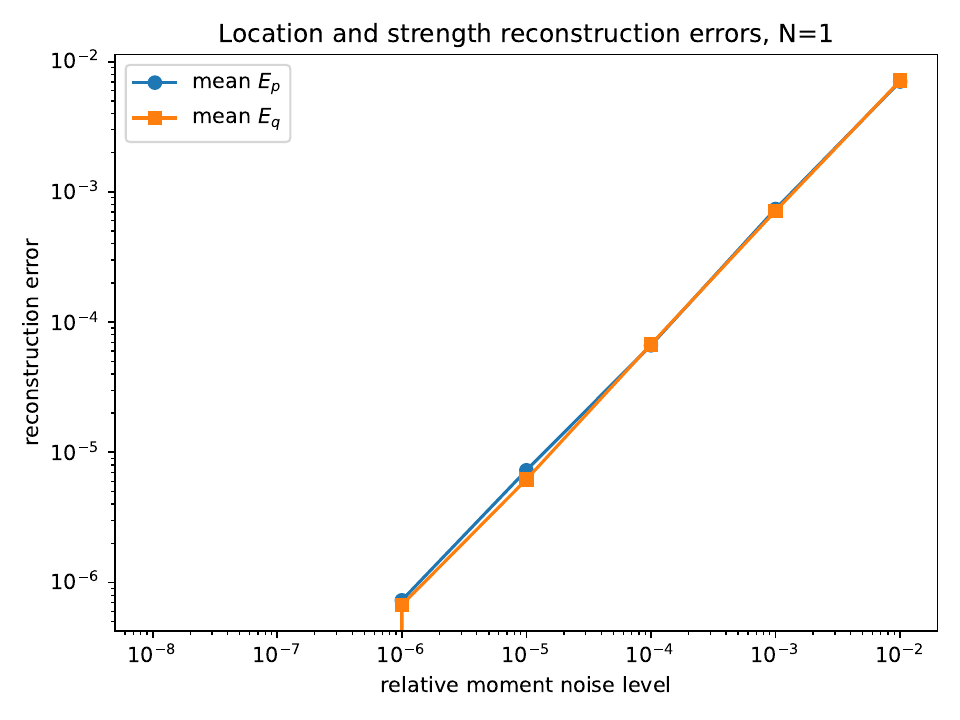}
\includegraphics[width=0.48\textwidth]{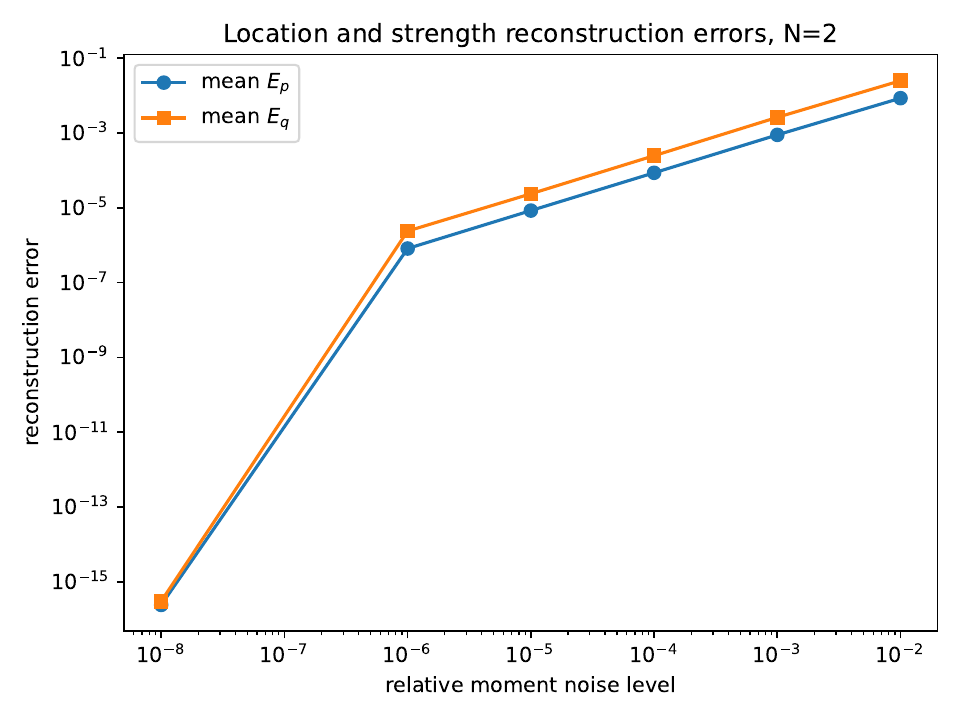}

\includegraphics[width=0.48\textwidth]{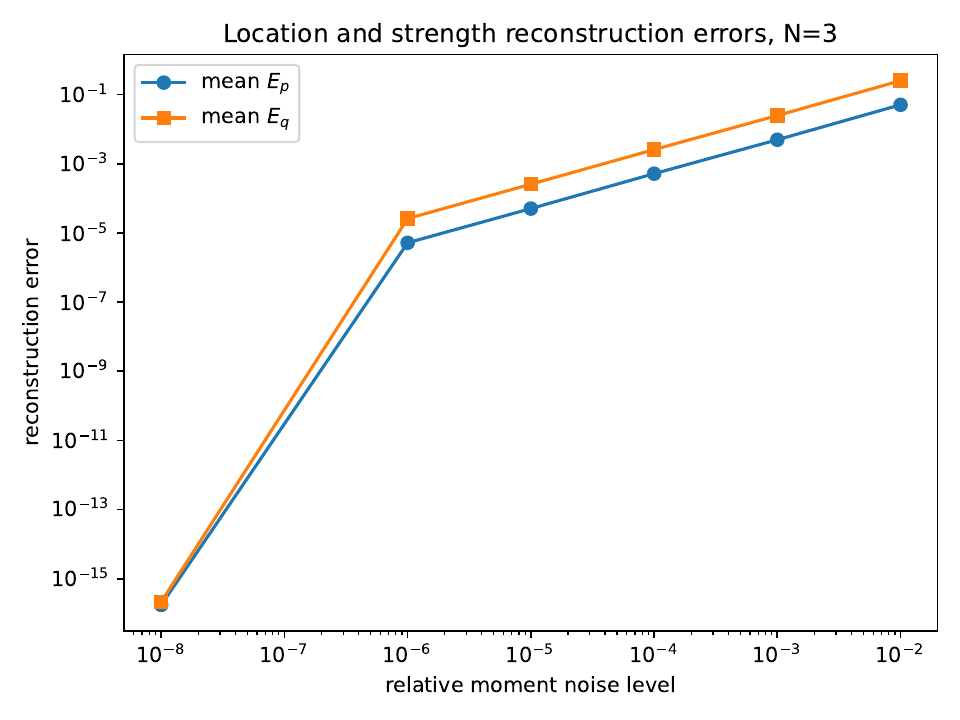}
\includegraphics[width=0.48\textwidth]{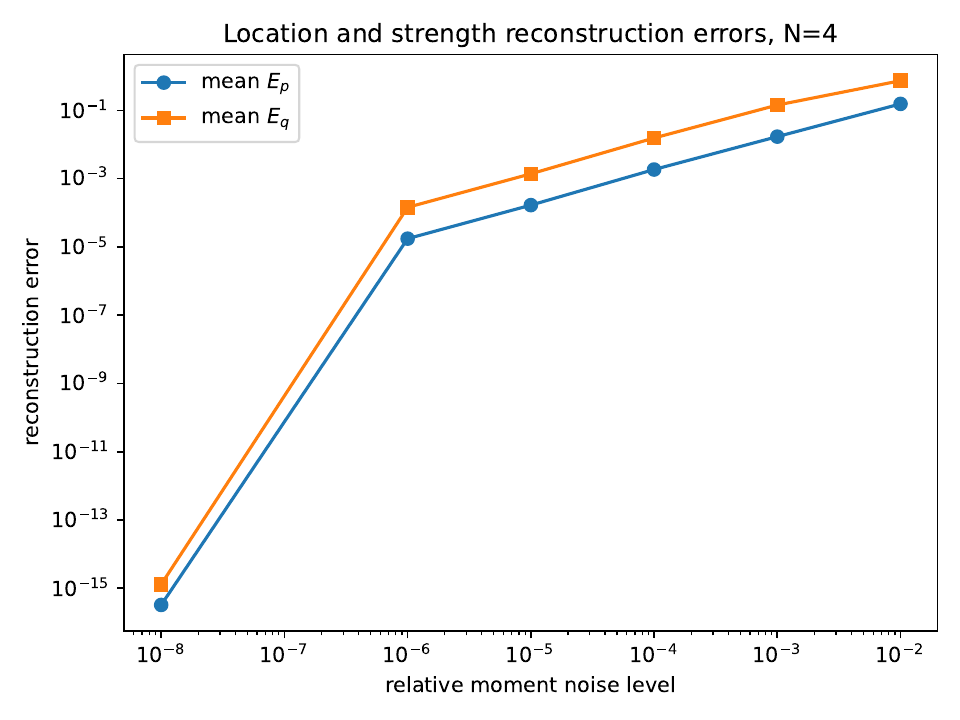}
\caption{Location and strength reconstruction errors for
\(N=1,2,3,4\). The errors increase with the relative perturbation level of the
low-frequency moments.}
\label{fig:reconstruction-errors}
\end{figure}

Figure~\ref{fig:vandermonde-conditioning} shows that the condition number of
the Vandermonde matrix increases with the source number:
\[
\kappa(V)\approx
1.00,\quad 3.06,\quad 8.83,\quad 31.99
\]
for \(N=1,2,3,4\), respectively. Therefore the post-processing step becomes
more sensitive as \(N\) increases. This is consistent with the classical
instability of Prony-type reconstruction and explains the growth of the
strength error for larger source numbers.

\begin{figure}[htbp]
\centering
\includegraphics[width=0.62\textwidth]{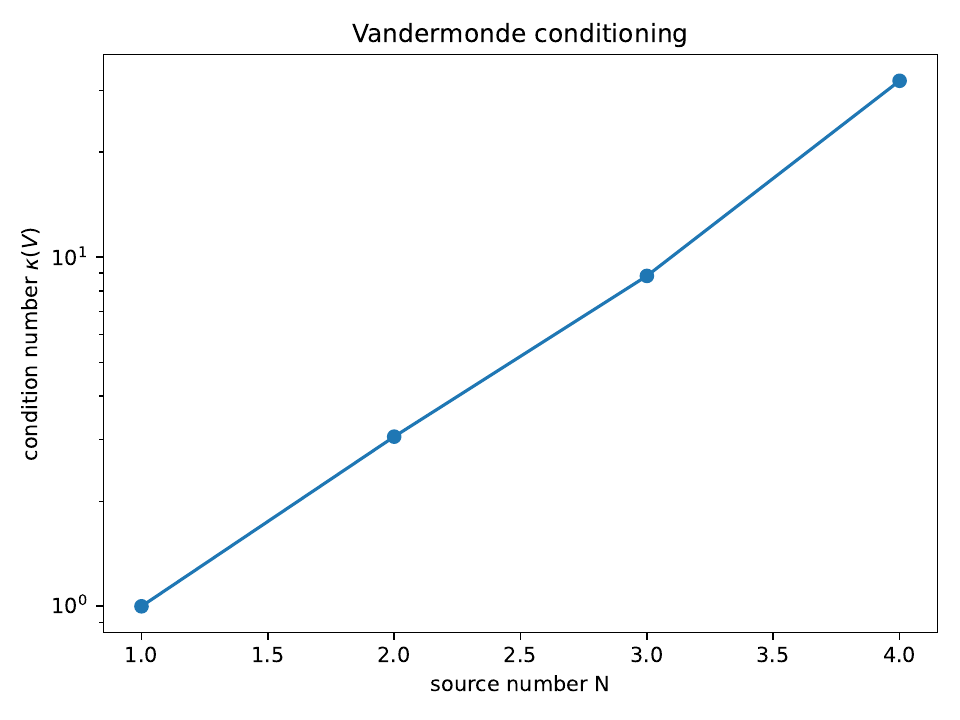}
\caption{Condition number of the Vandermonde matrix used in the recovery of
the source strengths. The conditioning deteriorates as \(N\) increases.}
\label{fig:vandermonde-conditioning}
\end{figure}

Overall, the numerical experiments confirm the main features of the proposed
framework. The Hankel determinant contour count exactly identifies the source
number in the noiseless case and is stable over a range of contour radii. The
Rouch\'e margin explains the sensitivity to noise, close sources, and weak
sources. Once the source number is determined, the same low-frequency moment
sequence can be used to recover the source locations and strengths through a
Prony--Vandermonde procedure, whose stability is governed by the conditioning
of the associated Vandermonde system.

\section{Conclusion}
\label{sec:conclusion}

This paper proposed a Hankel determinant approach for identifying an unknown
number of stationary point sources in a two-dimensional heat equation from
boundary flux data. In the unit disk, the Laplace-transformed and normalized
boundary flux admits a Fourier moment representation. The low-frequency limit
of this moment sequence has a finite exponential-sum structure, from which a
family of Hankel matrices and determinant characteristics can be constructed.

The main theoretical result shows that the number of point sources is encoded
in the vanishing order of the Hankel determinant characteristic at the zero
Laplace frequency. Under the distinctness of source locations, nonvanishing
source strengths, and a generic determinant lifting condition, the determinant
order satisfies
\[
\operatorname{ord}_{s=0}\Delta_m(s)
=
0
\quad (m\leq N),
\qquad
\operatorname{ord}_{s=0}\Delta_m(s)
=
m-N
\quad (m>N).
\]
Consequently, the source number is determined by the first Hankel order at
which the determinant contour count becomes nonzero. This gives a computable
source-counting formula through the argument principle.

We also established a Rouch\'e-type stability result for the determinant count.
The perturbation analysis shows how measurement noise, boundary discretization,
and finite-time truncation propagate from the boundary data to the Fourier
moments and then to the Hankel determinant. The resulting sufficient condition
has a clear interpretation: the determinant count is stable when the
determinant perturbation remains smaller than the Rouch\'e margin on the
chosen contour. This explains the sensitivity of the method in near-degenerate
situations, such as nearly colliding sources or very weak source strengths.

After the source number is identified, the same low-frequency moment sequence
can be used to recover the source locations and strengths. The locations are
obtained from an annihilating polynomial, and the strengths are then recovered
from a Vandermonde system. The numerical experiments confirm the predicted
count pattern in the noiseless case, show robustness with respect to the
contour radius, illustrate the decay of the Rouch\'e margin for close and weak
sources, and validate the subsequent location-strength recovery. They also
show that the strength recovery is more sensitive than the location recovery,
which is consistent with the conditioning of the Vandermonde system.

The present work focuses on the unit disk, where the Fourier moment
representation is explicit and the determinant structure can be analyzed
cleanly. Extending the construction to more general domains, developing
sharper non-degeneracy conditions for the determinant lifting mechanism, and
designing regularized moment estimators for strongly noisy data are natural
directions for future work.


%
%
%

\bibliographystyle{plain}   %
\bibliography{myref}       %

\end{document}